\documentclass[11pt,english]{article}
\usepackage[letterpaper]{geometry}
\usepackage{amsmath,amssymb,amsthm}
\usepackage{graphicx}
\usepackage{epstopdf}
\usepackage{setspace}
\usepackage{babel}

\theoremstyle{plain}
\newtheorem{theorem}{Theorem}
\newtheorem*{theorem*}{Theorem}

\newtheorem*{corollary*}{Corollary}
\newtheorem{definition}{Definition}
\newtheorem*{remark*}{Remark}

\def\mathscr{\mathfrak}

\def\R{{\mathbb R}}

\def\C{{\mathbb C}}

\def\CC{{\mathcal C}}

\def\ff{\widetilde{f}}

\def\-{\backslash}

\makeatletter

\date{}

\makeatother

\begin{document}

\title{Imaginary triangles, Pythagorean theorems, and algebraic geometry}

\author{Sergiy Koshkin\\
 Department of Mathematics and Statistics\\
 University of Houston-Downtown\\
 One Main Street\\
 Houston, TX 77002\\
 e-mail: koshkins@uhd.edu}
\maketitle
\begin{abstract}\
We extend the notion of triangle to ``imaginary triangles" with complex valued sides and angles, and parametrize families of such triangles by plane algebraic curves. We study in detail families of triangles with two commensurable angles, and apply the theory of plane Cremona transformations to find ``Pythagorean theorems" for them, which are interpreted as the implicit equations of their parametrizing curves.
 
\bigskip

\end{abstract}

\section{Introduction}\label{S1}

The use of algebraic geometry to study families of triangles has ancient roots. Back around 250 AD in the famous problem II.8 of his Arithmetica (on the margins of which Fermat wrote his famous comment) Diophantus introduced a trick for finding right triangles with three integer sides, the Pythagorean triples. This trick will later be interpreted (starting with Kronecker's 1901 algebra textbook, according to \cite{Lem}) as constructing a rational parametrization of the unit circle, an algebraic curve parametrizing the family of right triangles up to similarity. Around 940 AD Al Khazen proposed a problem of finding right triangles with rational sides and integer area, and proved that the latter is a congruent number. In modern terms, the problem amounts to finding rational points on some elliptic curves. Other families of triangles with rational sides are actively studied today using elliptic and higher genus curves \cite{Rus,SaSh}. 

In this paper we will apply to triangles the classical algebraic geometry of plane curves and Cremona transformations. It was developed in the works of 19th century authors such as Pl\"ucker, Cayley, Cremona, Clebsch and Max Noether before the onset of a more abstract modern approach after Hilbert, see \cite{Cob} for a historical survey, and \cite{Fult,Gib,SenWin,Wal} for modern introductions. It is attractive due to its more intuitive flavor, especially when applied to elementary geometry of triangles. Since algebraic geometry works best over the field of complex numbers it is helpful to expand the notion of triangle accordingly, hence the ``imaginary triangles" of the title.

Specifically, we will use algebraic geometry to study what we call ${p:q}$ triangles, defined similarly to the isosceles triangles, but with the base angles in an integer ratio ${p:q}$. As with the right triangles, parametrized by a conic, their parametrizing curves are rational. In fact, they can be parametrized by the Chebyshev polynomials (of the second kind), so we call them the Chebyshev curves. As with the right triangles, one can look for triples of integers that can be sides of ${p:q}$ triangles (``Pythagorean triples"), or for algebraic relations among those sides (``Pythagorean theorems"). These problems can be naturally interpreted as looking for rational points on, and implicit equations of, the Chebyshev curves. 

Despite the classical flavor of the problems to the best of the author's knowledge such triangle families were first studied only in 1954 by Oppenheim (${1:3}$ and ${2:3}$ cases, see \cite{Guy}). Later Oppenheim, together with Daykin, explicitly characterized primitive integer triples for all ${p:q}$ families, it seems fair to call them the Oppenheim triples. Their result was published in the Monthly back in 1967 \cite{DO}, but special cases and related results were rediscovered later multiple times, see \cite{Bick,CY,Desh,Hoyt,Kosh,Luth,Nic,Par,Selk,Will}. However, almost all attention went to the Oppenheim triples, while the ``Pythagorean theorems", and algebro-geometric connections, were largely overlooked. We hope to remedy this oversight. 

\section{Imaginary triangles and SSS}\label{S2}

Trigonometry literally translates from Greek as ``the measurement of triangles". But sine and cosine functions extend to complex values, and the trigonometric formulas continue to hold for them. It turns out that even some facts about the ordinary Euclidean triangles are best explained by looking at complex values. But what sorts of ``triangles" would have complex sides and angles?

The principal relations between sides and angles of the ordinary triangles are given by the laws of sines and cosines, so we should make sure that they continue to hold. Since trigonometric functions are $2\pi$ periodic, even for complex values, we should identify angles differing by a multiple of $2\pi$. Moreover, because cosines are even functions, and the overall sign change in the angles does not alter the law of sines, we should identify triples of angles differing by the overall sign change. This leads to the following definition.
\begin{definition} Let $\alpha,\beta,\gamma\in\C$ represent classes modulo $2\pi$ with $\alpha+\beta+\gamma=\pi\\ (\!\!\!\!\mod2\pi)$, and let $[-\alpha,-\beta,-\gamma]\sim[\alpha,\beta,\gamma]$. Denote by $\Lambda$ the resulting set of equivalence classes $[\alpha,\beta,\gamma]$. An {\bf imaginary triangle} (with ordered sides) is a pair $\big((a,b,c);[\alpha,\beta,\gamma]\big)\in\C^3\times\Lambda$ of sides and ``opposite" angles, such that the laws of sines and cosines hold for them. We call $c$ the base of the triangle, and $\alpha,\beta$ the {\bf base angles}.
\end{definition}
The imaginary triangles can be interpreted as living in $\C^2$ with sides and angles ``measured" using a bilinear form on it ({\it not} a Hermitian one), which extends the inner product on $\R^2$. This is nicely described in Kendig's paper \cite{Ken}. Our definition is slightly more refined since the bilinear form only defines complex valued sides up to sign. But even on our definition we can not get the side triples to cover all of $\C^3$. Suppose $c=0$, for example, then by the law of cosines $b^2=a^2$, and so $b=\pm a$. Any zero-side triangle must be either isoceles or ``anti-isosceles"! In particular, if an imaginary triangle has two zero sides then all three of them are zero. We will show, however, that this is the only restriction on the sides (Theorem \ref{SSS}).

But first let us look at zero-area triangles. Recall that $A=\frac12bc\sin\alpha=\frac12ac\sin\beta=\frac12ab\sin\gamma$ gives the area of an (ordinary) triangle. The law of sines $a:b:c=\sin\alpha:\sin\beta:\sin\gamma$ insures that all three expressions give the same value, even for imaginary triangles. But as long as we exclude the zero-side triangles, having
$\sin\alpha=0$, say, forces $\sin\beta=\sin\gamma=0$. In other words, zero-area triangles with non-zero sides can only have angles that are $0$ or $\pi$, see Fig.\ref{Patho}\,(a). But then by the law of cosines 
$a^2=b^2+c^2\pm2bc=(b\pm c)^2$ and $a\pm b\pm c=0$ for at least one choice of signs. This means that a triangle is zero-area if and only if
\begin{equation}\label{discrim}
\Delta:=(a+b+c)(-a+b+c)(a-b+c)(a+b-c)=4s(s-a)(s-b)(s-c)=0,
\end{equation}
where $s:=\frac12(a+b+c)$. One can show using the remaining laws of cosines, that the factors in \eqref{discrim} correspond to $[\pi,\pi,\pi]$, $[\pi,0,0]$, $[0,\pi,0]$ and $[0,0,\pi]$ angle triples, respectively. If the first of them looks impossible, recall that $3\pi=\pi\,(\!\!\!\!\mod2\pi)$. The second product in \eqref{discrim} should look familiar, it is $4$ times the expression under the square root in the ``Heron" area formula (likely due to Archimedes). So for ordinary triangles $\Delta=16A^2$. 

Now let us turn to the angles. Any pair $\alpha,\beta$ can serve as the base angles of an imaginary triangle, indeed $\gamma=\pi-\alpha-\beta$, $a=\sin\alpha$, $b=\sin\beta$, $c=\sin(\alpha+\beta)$ define one. However, this may produce a triangle with all sides equal to zero. This will not happen if at least one of $\alpha,\beta$ is neither $0$ nor $\pi$, and in that case the law of sines implies the law of cosines. To see this note that by the law of sines there is a $z$ such that $a=z\sin\alpha$, 
$b=z\sin\beta$, $c=z\sin(\alpha+\beta)$, and use the lesser known identity
\begin{equation}\label{SinSq}
\sin(\alpha+\beta)\sin(\alpha-\beta)=\sin^2(\alpha)-\sin^2(\beta).
\end{equation}
We will now prove a generalization of the side-side-side theorem (SSS) to imaginary triangles. The elementary SSS states that triangle's angles can be uniquely recovered from its sides, and gives a geometric construction of them. The uniqueness can not hold for the zero-side triangles though. If $c=0$, for example, and $a=b$ then $[\alpha,\pi-\alpha,0]$ would validate the law of sines for any $\alpha\in\C$ since $\sin\alpha=\sin(\pi-\alpha)$, see Fig.\,\ref{Patho}\,(a). And if $a=-b$ then $[\alpha,-\alpha,\pi]$ would do it. But even with the zero-side triangles excluded, we can not use the usual geometric constructions to prove SSS. Let us turn to complex analysis instead. 

\begin{figure}[!ht]
\vspace{-0.1in}
\begin{centering}
(a)\ \ \ \includegraphics[width=1.25in]{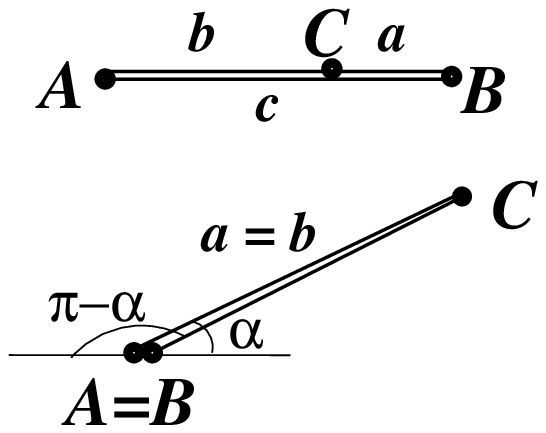} \hspace{0.2in} (b)\ \ \ \includegraphics[width=2.7in]{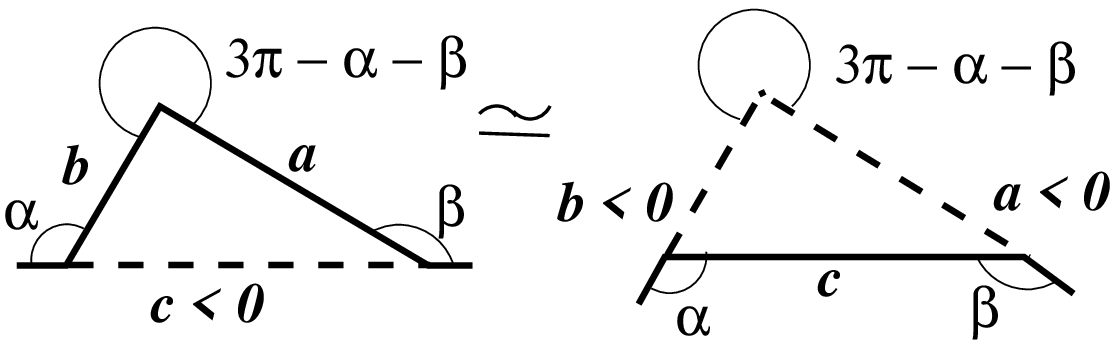}
\par\end{centering}
\vspace{-0.2in}
\hspace*{-0.1in}\caption{\label{Patho} (a) Zero-area and zero-side triangles (double lines are for visualization only); (b) Similar triangles with one and two negative sides.}
\end{figure}
Given three sides we can try to recover the angles by using the area formulas as equations: $\sin\alpha=\frac{2A}{bc}$ and $\sin\beta=\frac{2A}{ac}$. It follows from complex analysis that the system $\sin z=w$, $\cos z=\widetilde{w}$ has a solution $z\in\C$ if and only if $w^2+\widetilde{w}^2=1$, and this solution is unique modulo $2\pi$. In particular, $\sin z=w$ is solvable for any 
$w\in\C$, and the solution's class modulo $2\pi$ is uniquely determined by a choice of value for $\sqrt{1-w^2}$.
This implies that the angles (or rather their classes in $\Lambda$) are determined uniquely by non-zero sides. 
\begin{theorem}[{\bf SSS for imaginary triangles}]\label{SSS} Let $(a,b,c)\in\C^3$ and $a,b,c\neq0$. Then there is a unique $[\alpha,\beta,\gamma]\in\Lambda$  such that $\big((a,b,c);[\alpha,\beta,\gamma]\big)$ is an imaginary triangle. If one of 
$a,b,c$ is $0$ an imaginary triangle with these sides exists if and only if the other two are equal up to sign, and there are infinitely many such triangles.
\end{theorem}
\begin{proof} 
Given $a,b,c$ we compute $\Delta$ from \eqref{discrim}. There is a sign ambiguity in recovering $A$ from $\Delta$, but that is taken care of by specifying triples of angles only up to the overall sign change. Suppose first that $A\neq0$. There are at most two solutions, modulo $2\pi$, to each of the equations $\sin\alpha=\frac{2A}{bc}$ and $\sin\beta=\frac{2A}{ac}$. Moreover, they correspond to a choice of value for $\sqrt{1-\big(\frac{2A}{bc}\big)^2}$ and $\sqrt{1-\big(\frac{2A}{ac}\big)^2}$, respectively. One can see that the identity 
\begin{equation}\label{SqrtValues}
\frac{2A}{ac}\sqrt{1-\Big(\frac{2A}{bc}\Big)^2}+\frac{2A}{bc}\sqrt{1-\Big(\frac{2A}{ac}\Big)^2}=\frac{2A}{ab}.
\end{equation}
becomes equivalent to \eqref{discrim} by moving one of the square roots to the right, squaring, isolating the other square root, and squaring again. Working backwards, we find a unique choice of square roots that makes \eqref{SqrtValues} hold. Hence we can determine $\alpha,\beta$ uniquely, and \eqref{SqrtValues} then implies that
$$
\sin\alpha\cos\beta+\sin\beta\cos\alpha=\sin(\alpha+\beta)=\frac{2A}{ab}. 
$$
This verifies the law of sines (and, therefore, the law of cosines since $A\neq0$). Thus, $[\alpha,\beta,\pi-\alpha-\beta]$ are the sought angles, unique by construction.

Now consider triangles with $A=0$. Then one of the factors in \eqref{discrim} is zero. Suppose $a=b+c$, for example. Then $a^2=b^2+c^2+2bc$, and $bc\neq0$, so $\cos\alpha=-1$ to satisfy the law of cosines. This also implies $\sin\alpha=0$ and $\alpha=\pi$. The other two cosines are similarly found to be $1$, so $\beta=\gamma=0$. Conversely, the laws of sines and cosines hold with these assignments. The other cases are analogous.
\end{proof}
Triangles with negative sides are visualized on Fig.\ref{Patho}\,(b). They occupy the same ``place" as ordinary triangles with absolute values of their sides, but are ``viewed" differently where measuring the angles is concerned. However, triangles that violate the triangle inequalities, even with all positive sides like 1,1,3, have complex valued angles. To visualize them one has to step outside of $\R^2$, see \cite{Ken}.

The theorem means that, excluding the zero side triangles, we can uniquely parametrize imaginary triangles by points $(a,b,c)\in\C^3$ with three planes removed ($a=0$, $b=0$, and $c=0$). Two lines in each of those planes (intersections with 
$c=\pm b$, $c=\pm a$ and $b=\pm a$, respectively) have points corresponding to multiple zero-side triangles due to angle indeterminacy. The rest of the excluded planes corresponds to no triangles at all. We will be mostly interested in the shapes, or similarity classes, of triangles, that is we will identify those of them that are the same up to scale. The triangle shapes are parametrized by triple ratios $[a:b:c]$. Such triples, with $[0:0:0]$ excluded, which form the complex projective plane $\C P^2$, see e.g. \cite{Gib}. Planes through the origin in $\C^3$ become projective lines in $\C P^2$, and lines through the origin in $\C^3$ become projective points, $a,b,c$ are called the {\bf homogeneous coordinates}.
\begin{definition}\label{Shape} We call the {\bf triangle shape plane}, or simply the shape plane, the subset of $\C P^2$ obtained by removing three projective lines ($a=0$, $b=0$, and $c=0$) except for two points on each, ($[0:1:\pm1]$, $[1:0:\pm1]$, and $[1:\pm1:0]$). 
\end{definition}
\begin{figure}[!ht]
\vspace{-0.1in}
\begin{centering}
(a)\ \ \ \includegraphics[width=1.5in]{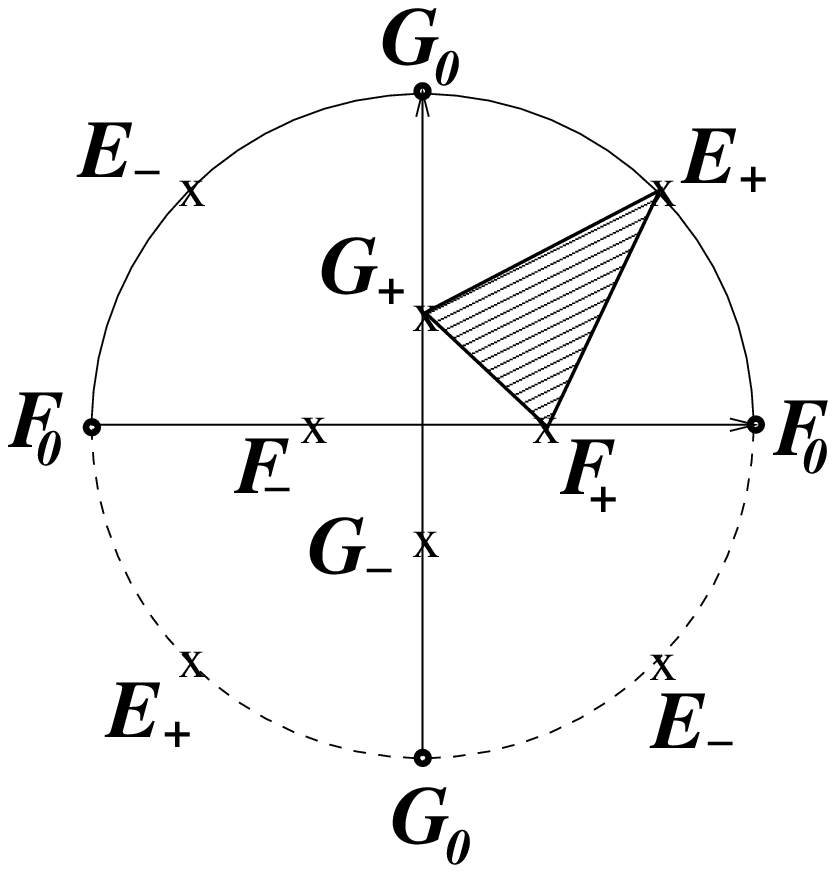} \hspace{0.7in} (b)\ \ \ \includegraphics[width=1.3in]{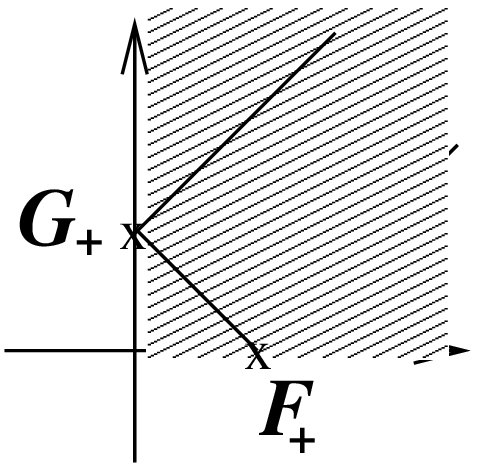}
\par\end{centering}
\vspace{-0.2in}
\hspace*{-0.1in}\caption{\label{RPGate} (a) Real projective plane of triangle shapes with the gate points and the Euclidean triangle, the lower semicircle is identified with the upper one; (b) Real affine plane of triangle shapes with the Euclidean half-strip.}
\end{figure}
To better understand $\C P^2$ it is instructive to look at its real slice, the real projective plane $\R P^2$. Most of it can be visualized by using the so-called affine coordinates, say $x:=\frac{b}{a}$, $y:=\frac{c}{a}$. These cover the ordinary plane $\R^2$, called the affine plane, and points with $a=0$ can be thought of as being at infinity. They form one of the lines excluded from the triangle shape plane, and the other two are exactly the $x$ and the $y$ axes. Two of the special points are on the line at infinity, and the four finite ones are $(\pm1,0)$ and $(0,\pm1)$. The points on the line at infinity represent directions of parallel lines in $\R^2$ (with opposite directions corresponding to the same point). We can visualize those points by adding to the plane an ``infinitely distant" circle ``enclosing" the entire affine plane. With antipodal points identified, it represents the line at infinity, see Fig.\,\ref{RPGate}\,(a).

One can see that each affine quadrant is actually a projective triangle with the coordinate half-axes and half of the line at infinity for sides. Moreover, each quadrant shares exactly one side with each of the other three, and there is exactly one point on each of those sides that is part of the real triangle shape plane. They correspond to the isosceles and the anti-isosceles zero-side triangles, and if not for them the real shape plane would be disconnected into four separate components.
\begin{definition}\label{GatePts} We denote $E_{\pm}:=[0:1:\pm1]$, $F_{\pm}:=[1:0:\pm1]$, $G_{\pm}:=[1:\pm1:0]$, and call them the {\bf gate points}. We also denote $E_0:=[1:0:0]$, $F_0:=[0:1:0]$, and $G_0:=[0:0:1]$.
\end{definition}
Note that even imaginary triangles with real sides may not be ``real" in the sense of elementary geometry, because they can have negative sides, or their sides may violate the triangle inequalities. Positivity of $a,b,c$ restricts us to the first quadrant, and rewriting the triangle inequalities in terms of $x$ and $y$ singles out a half-strip in it with $F_+$ and $G_+$ as corners, Fig.\,\ref{RPGate}\,(b). In the projective view the ``half-strip" is a projective triangle with the vertices 
$E_+$, $F_+$ and $G_+$, see Fig.\,\ref{RPGate}\,(a).

\section{Chebyshev curves}\label{S3}

The reason for introducing the triangle shape plane is that it streamlines the study of some natural families of triangles. The two most classical families are the right and the isosceles triangles, studied already by Pythagoreans. Both are defined by imposing a linear condition on the angles ($\alpha=\pi/2$ and $\beta=\alpha$, respectively), which implies an algebraic relation among their sides: the Pythagorean theorem for the right triangles, $a^2=b^2+c^2$, and the converse of Pons Asinorum for the isosceles ones, $b=a$. The corresponding curves in the affine shape plane are of the simplest kind, the unit circle and a vertical line, respectively. We will study the following generalizations of the isosceles family.
\begin{definition} Let $p,q$ be two relatively prime positive integers. We call an imaginary triangle $\boldsymbol{p:q}$ \textit{\textbf{triangle}} if ${\alpha:\beta}={p:q}$ for some angle representatives $\alpha,\beta$, or more precisely if $q\alpha-p\beta=0\ (\!\!\!\!\mod2\pi)$.
\end{definition}

Each pair $p:q$ defines a curve in the triangle shape plane. The corresponding curve can be parametrized using the law of sines. Indeed, if ${\alpha:\beta}={p:q}$ then $\alpha,\beta$ have a common measure $\theta\in\C$ such that $\alpha=p\theta$, $\beta=q\theta$ and $\gamma=\pi-(p+q)\theta$. The law of sines implies
$$
a:b:c=\sin(p\theta):\sin(q\theta):\sin(p+q)\theta,
$$
which is a parametrization of a curve in $\C P^2$ in homogeneous coordinates. For the purposes of algebraic geometry it is  ``bad" since it involves transcendental functions. But it can be transformed into a ``good" one using an observation of Chebyshev's \cite{CY} that the function $\frac{\sin(n+1)\theta}{\sin\theta}$ extends to a polynomial $U_n(\cos\theta)$, called the $n$-th \textit{\textbf{Chebyshev polynomial of the second kind}}. Moreover, $U_0=1$, $U_1=2t$ and $U_{n+1}=2tU_n-U_{n-1}$. 
The better known Chebyshev polynomials are those of the first kind, but here it is the second kind ones that take the center stage. Rewriting the law of sines in terms of $U_n$ gives the parametrization (with $t=\cos\theta$):
\begin{equation}\label{Upq}
a:b:c=U_{p-1}(t):U_{q-1}(t):U_{p+q-1}(t),\ t\in\C\,.
\end{equation}
In fact, it is more natural to let $t$ take values not in $\C$ but in $\C\cup\{\infty\}$. The $\infty$ maps to $G_0=[0:0:1]\in\C P^2$. This is the only point on (the closure of) our curve that does not correspond to any, even imaginary, triangle. And aside from the gate points it is the only point that all $p:q$ curves pass through, one could even say that they all ``begin" and ``end" at it, see Fig.\,\ref{ProjDiag}.
\begin{definition} We denote by $\CC_{p,q}$ the closure in $\C P^2$ of the curve parametrized in \eqref{Upq}, and call it the  $\boldsymbol{p:q}$ \textit{\textbf{Chebyshev curve}}. The point $G_0=[0:0:1]\in\CC_{p,q}$ is called the \textit{\textbf{source point}}.
\end{definition}
This definition will suffice for now, but we will later refine it to treat $\CC_{p,q}$ as algebraic curves rather than just parametrized point sets (Definition \ref{AlgCpq}). Note that because $U_n$ have integer coefficients rational values of $t$ are mapped into rational points (that is points with rational coordinates), and therefore produce integer triples that can be sides of (imaginary) triangles. It turns out that the converse is also true, irrational values of $t$ do not map into rational points. As noticed in \cite{CY}, this can be shown as follows. The law of cosines expresses $\cos(p\theta)$, $\cos(q\theta)$ and $\cos(p+q)\theta$ as rational functions of $a,b,c$ with $ab$, $ac$ and $bc$ in the denominators. One can then use Chebyshev polynomials (of both kinds) to express $t=\cos\theta$ as a polynomial in them. We leave the details as an exercise to the reader.

It follows that the only potential self-intersection points of $\CC_{p,q}$, that is points that distinct values of $t$ are mapped into, are the points where $a$, $b$ or $c$ are $0$, in other words, the gate points $E_{\pm}$, $F_{\pm}$ and $G_{\pm}$ (the source point $G_0$ can only correspond to $t=\infty$ because it does not represent a triangle). The trace of $\CC_{p,q}$ in $\R P^2$ is qualitatively determined by the order in which they are visited, as we shall see. The next theorem will allow us to find that order, as well as the self-intersection multiplicities of the gate points. 
\begin{theorem}\label{ParRec} As $t\in\R$ grows from $-\infty$ to $\infty$ on the Chebyshev curve $\CC_{p,q}$ the gate points $E_{\pm}$ are passed $p-1$ times ending with $E_-$ (at $t=\cos\frac{\pi}p$), $F_{\pm}$ are passed $q-1$ times ending with $F_-$ (at $t=\cos\frac{\pi}q$), and $G_{\pm}$ are passed $p+q-1$ times ending with $G_+$ (at $t=\cos\frac{\pi}{p+q}$). These are the only self-intersection points of $\CC_{p,q}$ (if any), and their multiplicities (number of distinct parameter values mapped into them) are given in Table \ref{GateMult}. When $t\to\pm\infty$ all Chebyshev curves approach the source point $G_0$.
\end{theorem}
\begin{table}
\begin{centering}
  \begin{tabular}{| c | c | c | c | c | c | c |}
    \hline
          Parities    & $E_+$         &   $E_-$       & $F_+$         &     $F_-$     &      $G_+$    & $G_-$\\ \hline 
    $p$ odd, $q$ odd  & $\frac{p-1}2$ & $\frac{p-1}2$ & $\frac{q-1}2$ & $\frac{q-1}2$ & $\frac{p+q}2$ & $\frac{p+q}2-1$\\ 
    $p$ odd, $q$ even & $\frac{p-1}2$ & $\frac{p-1}2$ & $\frac{q}2-1$ & $\frac{q}2$ & $\frac{p+q-1}2$ & $\frac{p+q-1}2$\\
    $p$ even, $q$ odd & $\frac{p}2-1$ & $\frac{p}2$ & $\frac{q-1}2$ & $\frac{q-1}2$ & $\frac{p+q-1}2$ & $\frac{p+q-1}2$\\ \hline
  \end{tabular}
\par\end{centering}
\caption{Real self-intersection multiplicities of the gate points.}\label{GateMult}
\end{table}
\begin{proof}Recall that $E_\pm=[0:1:\pm1]$, so by the parametrization \eqref{Upq} $E_{\pm}$ is on the curve for $t$ such that
$U_{p-1}(t)=\frac{\sin(p\theta)}{\sin(\theta)}=0$.
This means that $p\theta=\pi k$ for integer $k$ while $\theta\neq\pi j$. Therefore $\theta_k=\frac{\pi k}p$ for $k=1,\dots,p-1$ produces all the possible values for $t_k=\cos\theta_k$. Note also that $\sin(p+q)\theta_k=(-1)^k\sin(q\theta_k)$, and therefore $U_{p+q-1}(t_k)=(-1)^kU_{q-1}(t_k)$ according to \eqref{Upq}. In other words, as $k$ increases from $1$ to $p-1$ the passages through $E_+$ and $E_-$ interlace, with $E_-$ being first in $k$, but last in $t$. This gives the first two columns of Table \ref{GateMult}. The other four are analogous.
\end{proof}
Once the order of passage through the gate points has been determined the real projective trace of $\CC_{p,q}$ can be easily sketched. It is convenient to first find the ``code" of a curve, the list of gate points in the order of passage.
\begin{figure}[!ht]
\begin{centering}
(a)\ \ \ \ \includegraphics[scale=0.5]{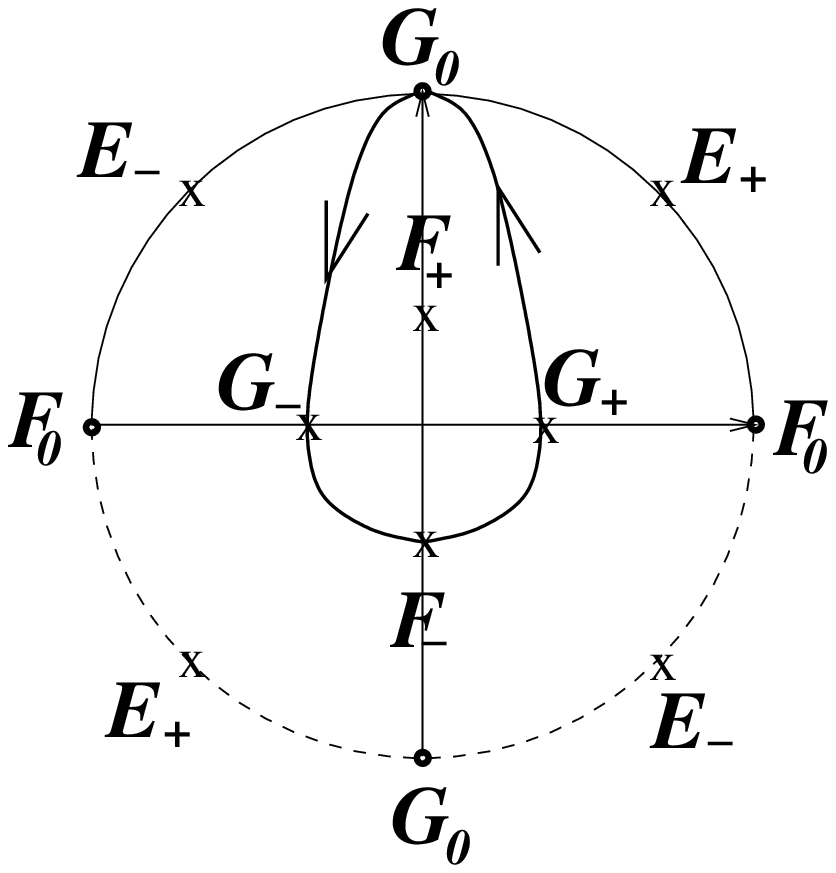} \hspace{0.5in} (b)\ \ \ \ \includegraphics[scale=0.5]{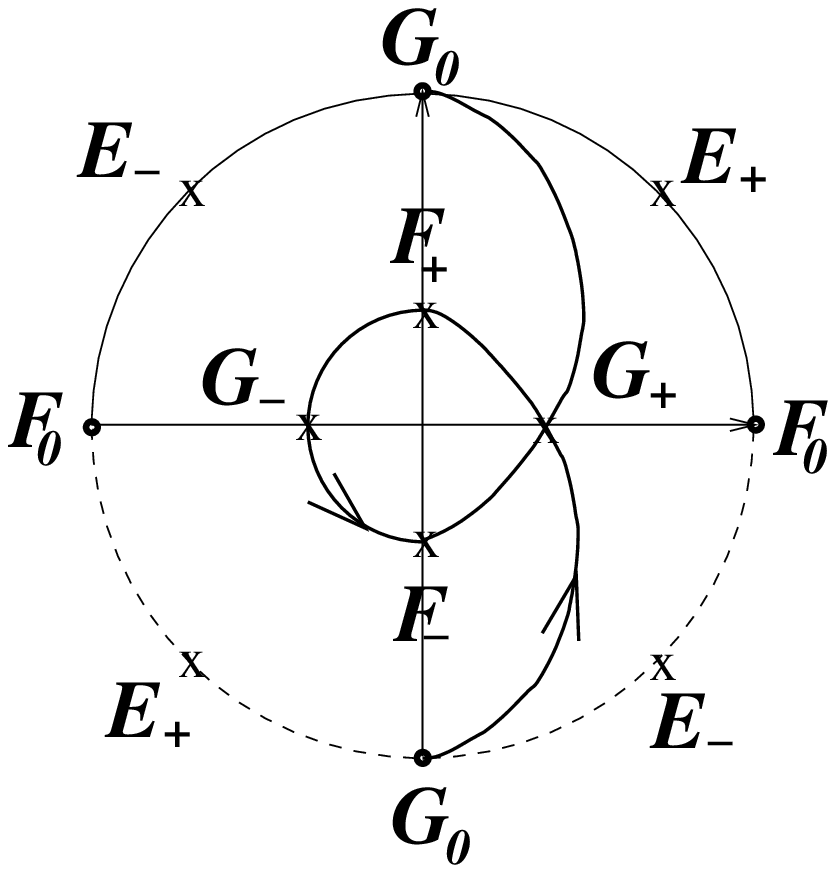}

(c)\ \ \,\ \ \ \ \includegraphics[scale=0.5]{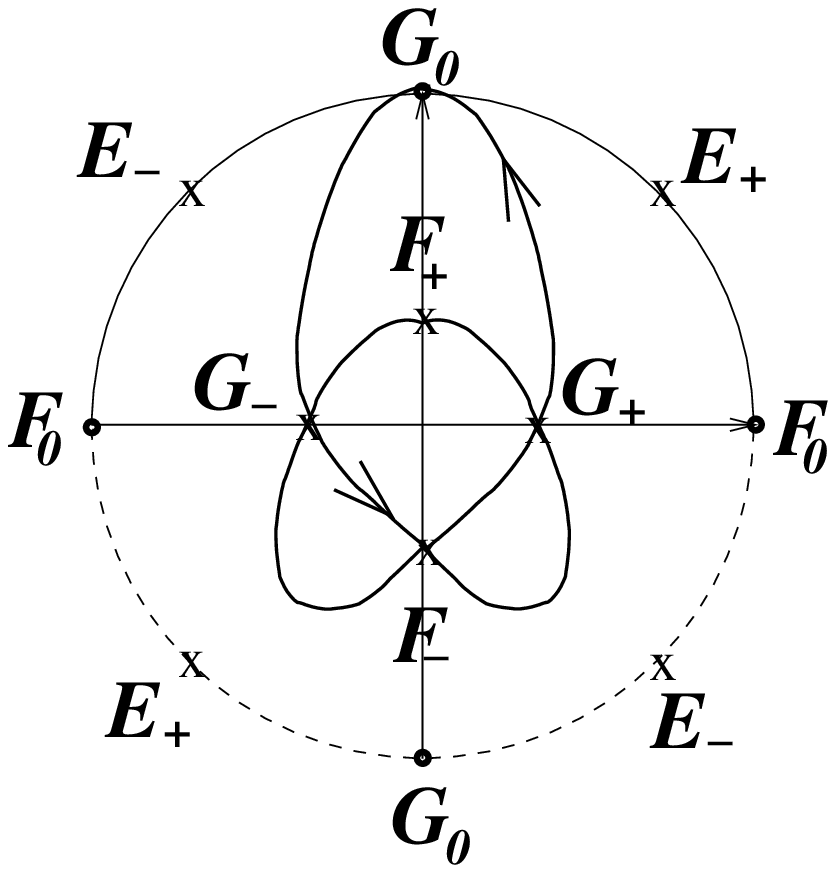} \hspace{0.5in} (d)\ \ \ \ \includegraphics[scale=0.5]{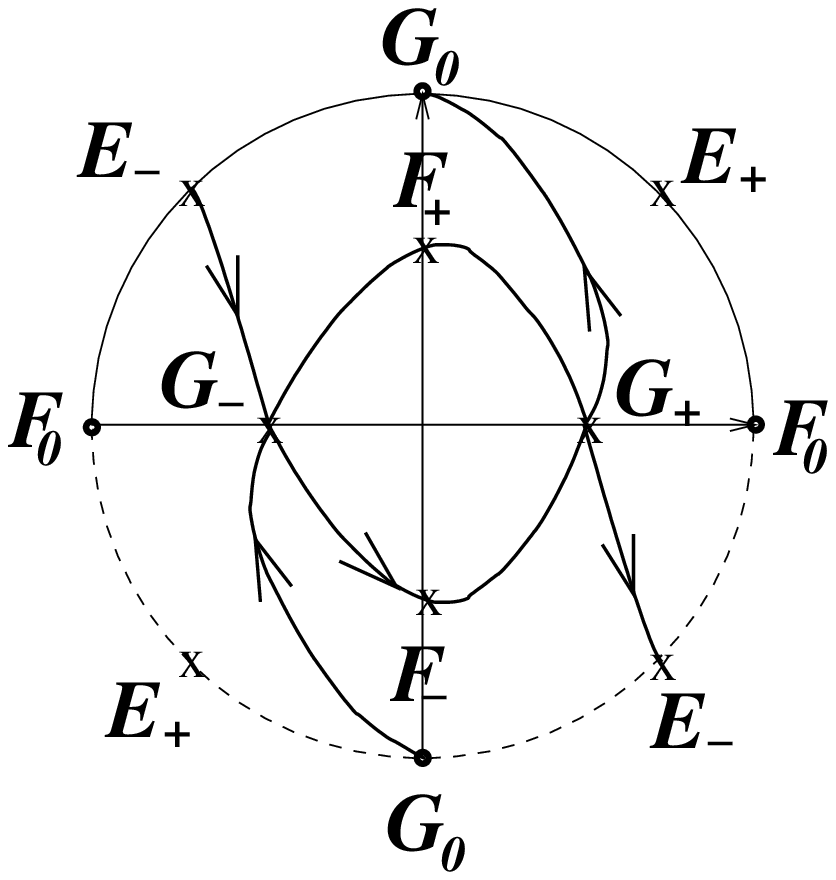}
\par\end{centering}
\caption{\label{ProjDiag}Real projective diagrams of (a) $1:2$, (b) $1:3$ (c) $1:4$, and (d) $2:3$ curves.}
\end{figure}

{\bf Examples:} For the $1:2$ curve $E_\pm$, $F_+$ are not passed at all, $F_-$ is passed at $\theta=\frac\pi2$ and $G_\pm$ at $\theta=\frac\pi3, \frac{2\pi}3$. We have the table:
$$
E_\pm: \emptyset\,;\ \ \ \ F_\pm: \frac\pi2 -\,, \ \ \ \ G_\pm: \frac\pi3 +, \frac{2\pi}3 -\,;
$$
where the signs are easily assigned since they alternate, and we know the first sign for each pair of points. So the code is 
$G_0G_-F_-G_+G_0$ in the increasing order of $t$ (decreasing of $\theta$), see Fig.\ref{ProjDiag}. Similarly, for the $1:3$ curve the code is $G_0G_+F_+G_-F_-G_+G_0$, and for the $2:3$ curve it is $G_0G_-F_+G_+E_-G_-F_-G_+G_0$. Note that in the affine view $x=\frac{b}{a}$, $y=\frac{c}{a}$ curves passing through $E_\pm$ will look like having separate branches (intersecting at the other gate points), and approaching $G_0$ will look like approaching the vertical direction towards infinity (although not a vertical asymptote), see Fig. \ref{MapleGraphs}.
\begin{figure}[!ht]
\begin{centering}
(a)\ \ \ \ \includegraphics[scale=0.15]{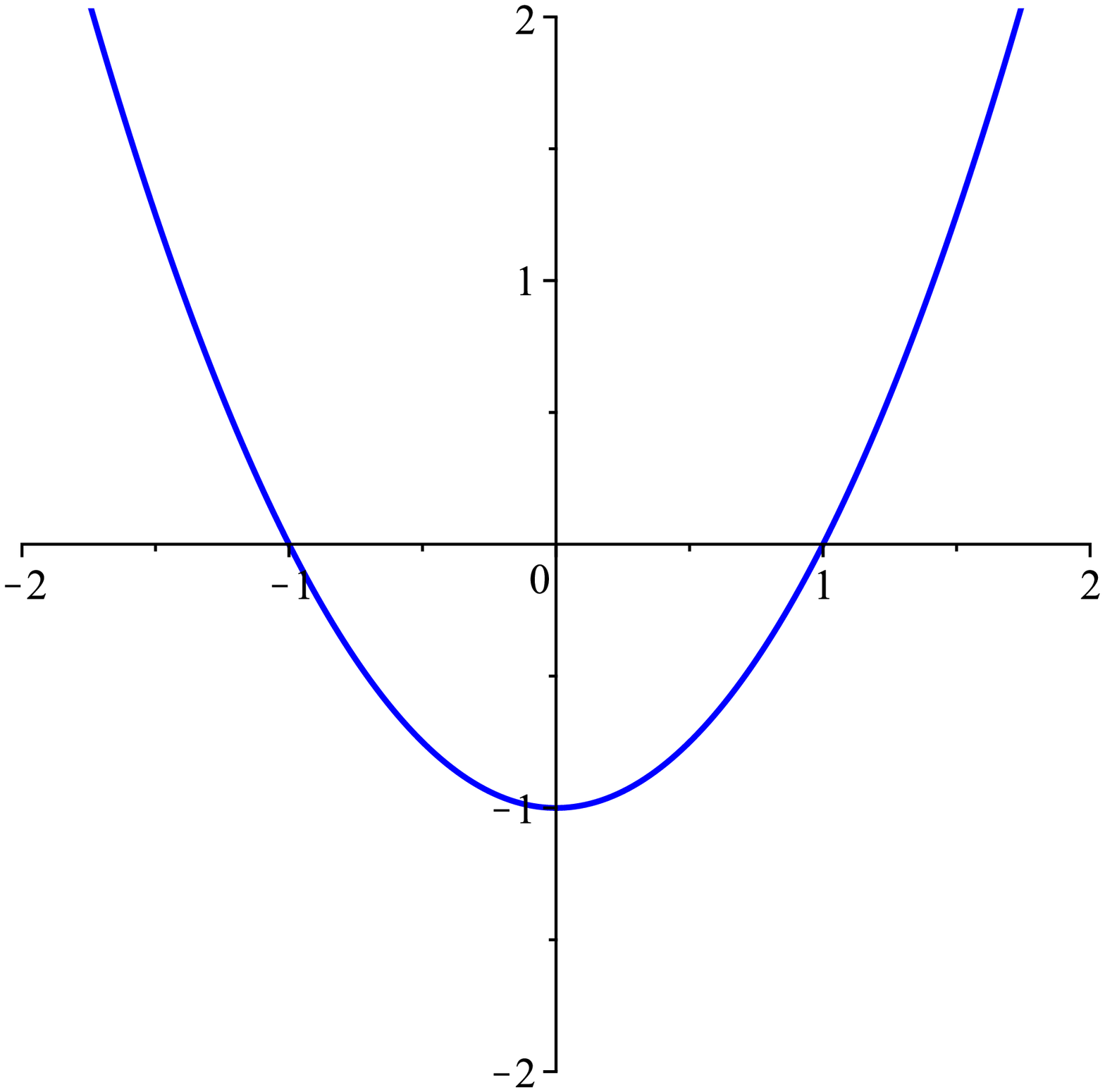} \hspace{0.2in} (b)\ \ \ \ \includegraphics[scale=0.15]{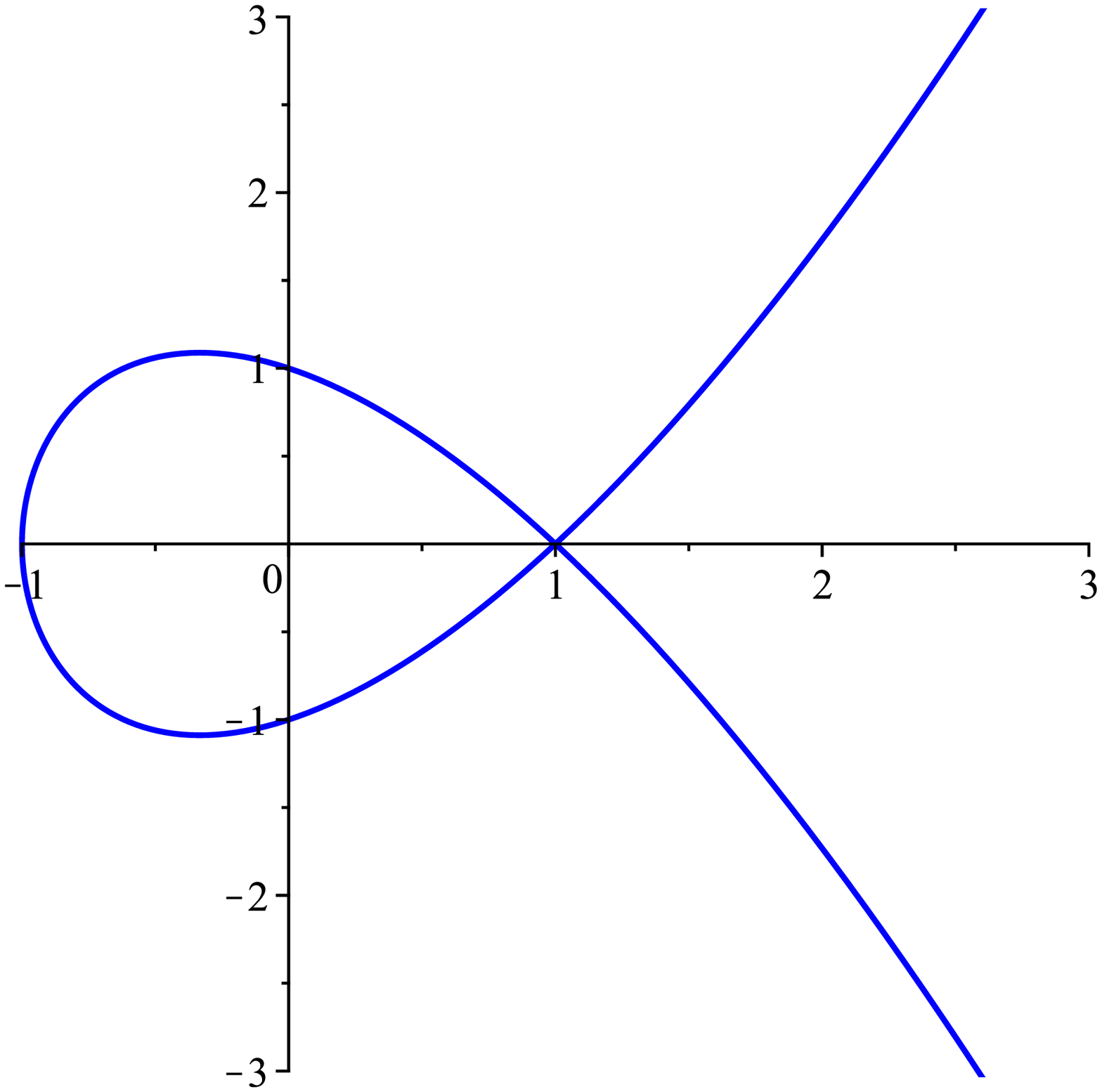}
(c)\ \ \ \  \includegraphics[scale=0.15]{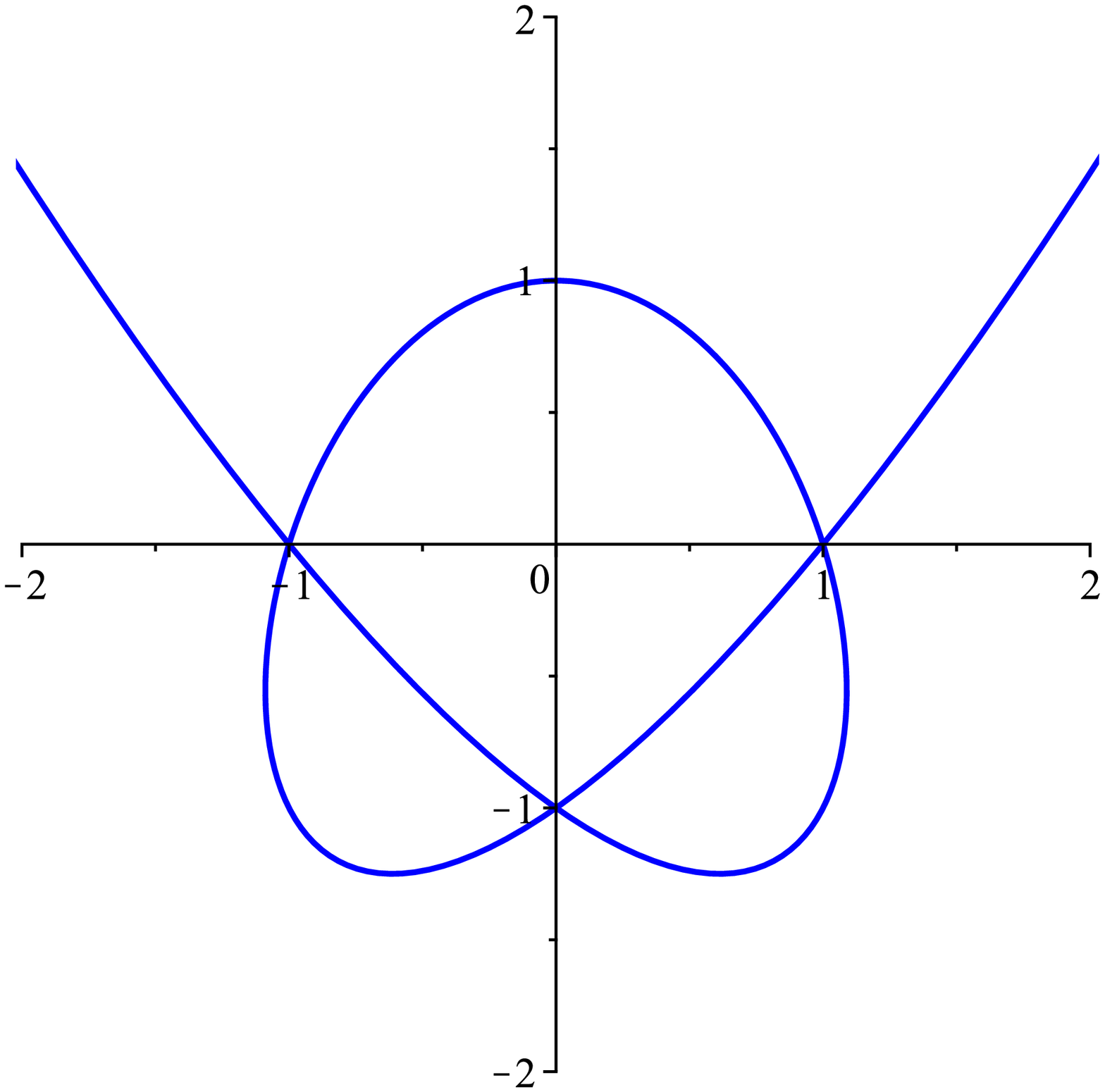} \hspace{0.2in}\\ (d)\ \ \ \ \includegraphics[scale=0.15]{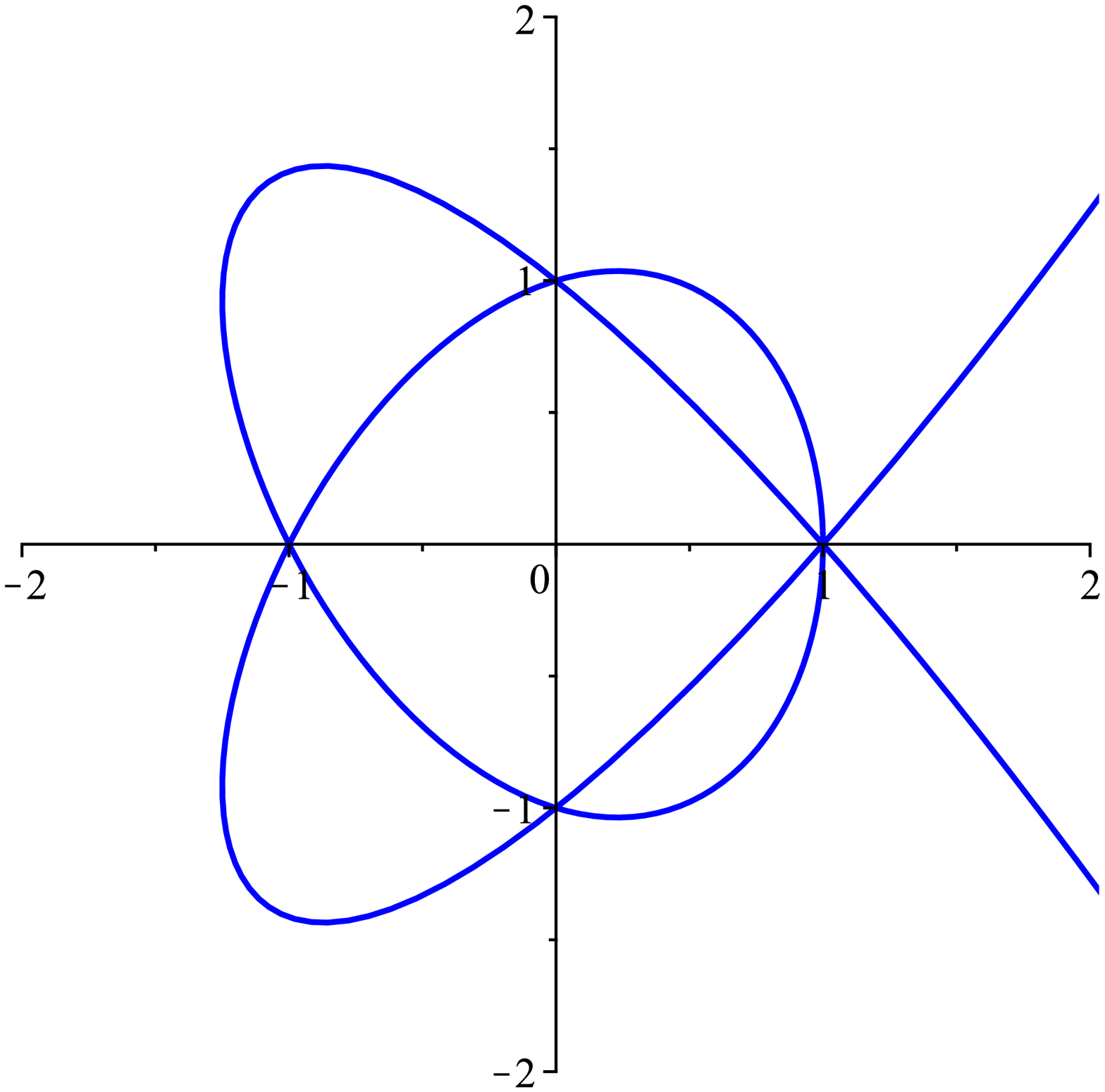}
(e)\ \ \ \  \includegraphics[scale=0.15]{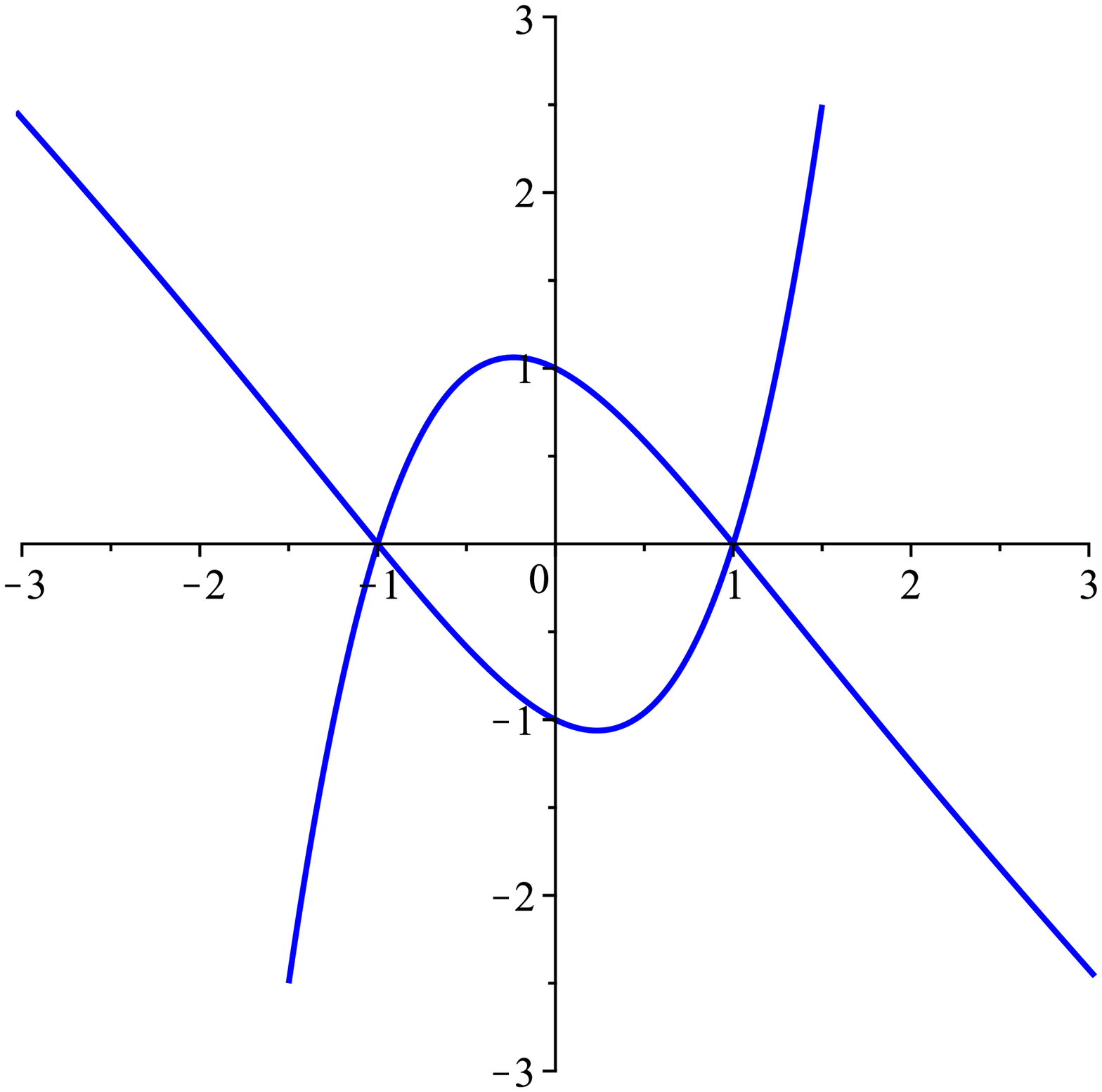} \hspace{0.2in} (f)\ \ \ \ \includegraphics[scale=0.15]{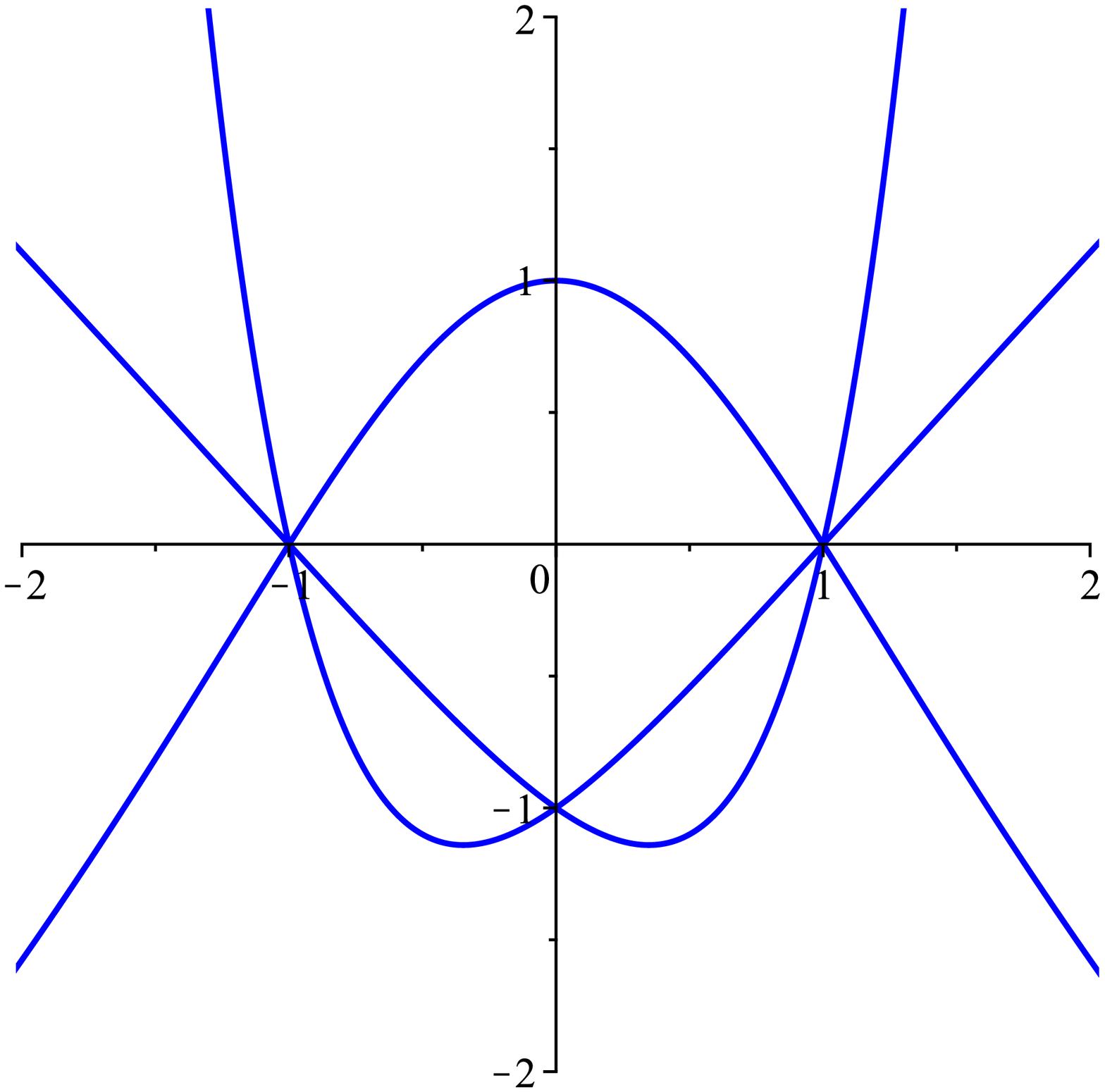}
\par\end{centering}
\caption{\label{MapleGraphs}Real affine graphs of (a) $1:2$ (b) $1:3$ (c) $1:4$ (d) $1:5$ (e) $2:3$ and (f) $3:4$ curves, drawn by Maple.}
\end{figure}

\section{Heron transformation and its inverse}\label{S4}

Parametrizations are central in differential geometry, but in algebraic geometry implicit equations are more commonly used. The Pythagorean theorem can be written as a homogeneous implicit equation $b^2+c^2-a^2=0$ that defines an algebraic curve in 
$\C P^2$ with $a,b,c$ as homogeneous coordinates. This curve is a non-degenerate conic, which can be rendered as an ellipse (circle) or a hyperbola, depending on a choice of affine coordinates. With the choice we are using, $x=\frac{b}{a}$, $y=\frac{c}{a}$, it becomes the unit circle $x^2+y^2-1=0$. We can expect something similar for our $p:q$ triangles. In algebraic terms, we are looking for implicit equations $f_{p,q}(a,b,c)=0$ of the Chebyshev curves $\CC_{p,q}$. For $\CC_{1,2}$ it can be found by inspection from the parametrization \eqref{Upq}: we have $\frac{b}{a}=t$ and $\frac{c}{a}=t^2-1$, so $b^2-a^2-ac=0$. But already for $\CC_{1,3}$ the equation is not so obvious. Fortunately, there is a recursion for computing $f_{p,q}$ with nice interpretations in both elementary and algebraic geometry. It involves transforming triangles with the base angles $\alpha,\beta$ into those with $\alpha,\beta-\alpha$. 

We start with the elementary interpretation. Consider a triangle $\triangle ABC$ with the base angles $\alpha$, $\beta$, and from the vertex $C$ opposite the base drop a segment $CB'$ on it so that $\triangle B'CB$ is isosceles, Fig. \ref{Heron}. By inspection, $\triangle ACB'$ has the base angles $\alpha$, $\beta-\alpha$, which means that we can transform $p:q$ triangles into $p:q-p$ ones this way.
\begin{figure}[!ht]
\vspace{-0.05in}
\begin{centering}
\includegraphics[width=1.7in]{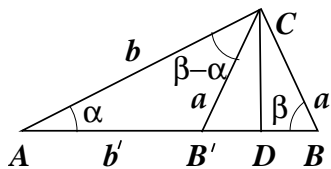} 
\par\end{centering}
\vspace{-0.3in}
\hspace*{-0.1in}\caption{\label{Heron} Heron transformation of an ordinary triangle.}
\end{figure}
Let us see how the transformation acts on the sides $a,b,c$ of $\triangle ABC$ opposite $A,B,C$. If $a',b',c'$ are the corresponding sides of $\triangle ACB'$ then $a'=a$ and $c'=b$ by inspection. To find $b'$ drop the perpendicular $CD$ to $AB$ and apply the Pythagorean theorem twice, to $\triangle ACD$ and $\triangle BCD$. This gives $b'=\frac{b^2-a^2}c$. Since we only care about triangles up to scaling we can homogenize the transformation, i.e. take $ac$, $b^2-a^2$ and $bc$ as the new sides.
\begin{definition} The {\bf Heron transformation} is the rational map $H:\C P^2\dashrightarrow\C P^2$ given in homogeneous coordinates by $H([a:b:c])=[ac:b^2-a^2:bc]$. We will occasionally abuse notation by denoting the (non-homogenized) Heron transformation also by $H:\C^3\to\C^3$.
\end{definition}
The above elementary construction of $H$ appears in \cite{Bick,Par,Selk}, all three papers use it to find the Oppenheim triples recursively. R.G. Rogers pointed out to the author that the construction is reminiscent of Heron's in his Metrica, hence the name.

Now let us turn to algebraic geometry. The reason for the dashed arrow in the definition is that $H$ is not defined on quite all of $\C P^2$. It is of course always defined as a map $\C^3\to\C^3$, but if the image of a point is $(0,0,0)$ then there is no point in $\C P^2$ that corresponds to it since $[0:0:0]$ is not in $\C P^2$. One can see that this happens if and only if $a=b=0$ or $c=0$ and $b=\pm a$. Thus, the source point $G_0=[0:0:1]$ and the gate points $G_\pm=[1:\pm1:0]$ are the only points of $\C P^2$ not in the domain of $H$. But $f\circ H$ is everywhere defined (for non-homogenized $H$), and it is a homogeneous polynomial if $f$ is. 

$H$ is not one-to-one or onto either. So it may come as a surprise that despite all that it has an explicit (almost) inverse, and a rational one at that, no square roots! The reason is that on $\C P^2$ we only invert up to scale. Indeed, if $H([a:b:c])=[u:v:w]$ then $\frac{b}{a}=\frac{w}{u}$ and $\frac{b+a}{c}=\frac{v}{w-u}$, so $\frac{c}{a}=\frac{w^2-u^2}{uv}$. 
\begin{definition} The {\bf inverse Heron transformation} is the rational map $H^{-1}:\C P^2\dashrightarrow\C P^2$ given in homogeneous coordinates by $H([u:v:w])=[uv:vw:w^2-u^2]$.
\end{definition}
The ``inverse Heron transformation" should be understood as an idiom, it is not a set-theoretic inverse of $H$ on $\C P^2$. In elementary terms, it transforms triangles with the base angles $\alpha$, $\beta$ into ones with the base angles $\alpha$, $\beta+\alpha$. Like $H$ it is undefined at three points, namely $F_0=[0:1:0]$ and $F_\pm=[1:0:\pm1]$. When we look for their pre-images under $H$ we find that those are entire lines $c=0$ and $b=\pm a$, respectively. Reciprocally, $H^{-1}$ maps the lines $v=0$ and $w=\pm u$ into the points where $H$ is undefined, $G_0$ and $G_\pm$. So ``undefined" should not be taken to mean that those points go nowhere. One can even show that when a curve approaches one of them, its image under $H$ approaches a point on the corresponding line, different tangents of approach corresponding to different points \cite[II.I.1]{Cool}. It is said that $H$ and $H^{-1}$ {\it blow up} three points into three lines, and {\it blow down} three lines into three points \cite[3.2]{SenWin}, see Fig.\,\ref{Blow}.

$H$ is an example of quadratic {\bf plane Cremona transformation}. The points where it is undefined are called its {\bf base points}, and the curves (lines) that it maps into points its {\bf exceptional lines} \cite{Alb}. Points on the exceptional lines are called exceptional and form the {\bf exceptional locus}.
\begin{figure}[!ht]
\begin{centering}
\includegraphics[scale=0.5]{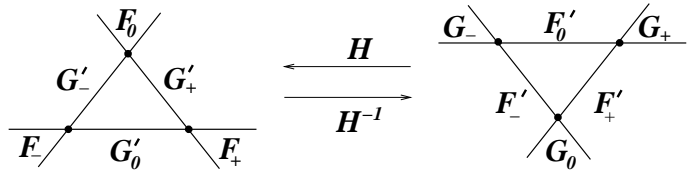}
\par\end{centering}
\caption{\label{Blow}Heron transformation diagram: blow ups and blow downs.}
\end{figure}
In classical works the base points are called fundamental, and the exceptional curves fundamental or principal 
\cite{Cob,Cool}. So $H$ blows up its base points into the exceptional lines of $H^{-1}$, and blows down its exceptional lines into the base points of $H^{-1}$. The action of $H$ and $H^{-1}$ is depicted schematically on Fig. \ref{Blow}, where the lines are labeled by (primed) points they are blown down to. The two maps are the set-theoretic inverses of each other on $\C P^2$-s with the exceptional loci removed.

If $[a:b:c]\in\C P^2$ represents an imaginary triangle with the base angles $\alpha,\beta$, and is not a gate point or exceptional for $H$ ($b\neq\pm a$), then $H([a:b:c])$ represents an imaginary triangle with the base angles $\alpha,\beta-\alpha$. Analogously, if $[a:b:c]$ is non-gate and non-exceptional for $H^{-1}$ then $H^{-1}([a:b:c])$ represents an imaginary triangle with the base angles $\alpha,\beta+\alpha$. In generic cases this can be seen by noticing that if $a:b:c=\sin\alpha:\sin\beta:\sin(\alpha+\beta)$ and $[u:v:w]=[ac:b^2-a^2:bc]$, then $u:w=ac:bc=a:b=\sin\alpha:\sin\beta$, and $u:v=\sin\alpha:\sin(\beta-\alpha)$ due to identity \eqref{SinSq}. We leave checking the degenerate cases with base angles $(0,0)$, $(0,\pi)$, $(\pi,0)$ or $(\pi,\pi)$ to the reader. 

Now we can reason as follows. Suppose we know the ``Pythagorean theorem" (implicit equation) for $p:q-p$ tringles, say 
$f_{p,q-p}=0$. Then we should have $f_{p,q-p}(a',b',c')=f_{p,q-p}(a,\frac{b^2-a^2}c,b)=0$, and since $f_{p,q-p}$ is a homogeneous polynomial the composition with it should still vanish on $\CC_{p,q}$. Unfortunately, this does not imply the converse, that any triple satisfying this equation represents the sides of a $p:q$ triangle. Indeed, the naive recursion introduces extraneous factors that have to be dealt with. For instance,
$$
f_{1,2}\circ H=(b+a)\Big((b+a)(b-a)^2-ac^2\Big)=(b+a)f_{1,3},
$$
similarly $f_{1,3}\circ H=(b+a)(b-a)f_{1,4}$. As one may suspect, it is the behavior of $H$ at and near the exceptional locus that is responsible for these extraneous factors. To explain it we need some more terminology \cite[7.4]{SenWin}, \cite[III.7.3]{Wal}. 

Up to now we essentially identified algebraic curves with sets of points. But they are more than sets, $f=0$ and $f^2=0$ define the same set of points in $\C P^2$, but they are different curves. As understood in algebraic geometry, they differ in ``multiplicity". In fact, in algebraic geometry a curve is identified with the polynomial (up to a numerical multiple) that defines it, but notationally it is still convenient to distinguish between the polynomial $f$ and the curve $\CC_f$ as a geometric object. 
\begin{definition}\label{AlgCpq}
Let $\Gamma_{p,q}(t)$ denote the Chebyshev parametrization \eqref{Upq}, and $f_{p,q}$ be the homogeneous polynomial of minimal degree such that $f_{p,q}\circ\Gamma_{p,q}=0$ as a polynomial in $t$. From now on we refine $\CC_{p,q}$ to mean the algebraic curve $\CC_{f_{p,q}}$.
\end{definition}
If $f$ has proper factors the curves they define are called {\bf components} of $\CC_f$. Polynomials that can not be factored non-trivially, and the curves they define, are called {\bf irreducible}. One can show that $f_{p,q}$ is well-defined, irreducible, and of degree $p+q-1$ \cite[4.1]{SenWin}.
\begin{figure}[!ht]
\begin{centering}
\includegraphics[scale=0.5]{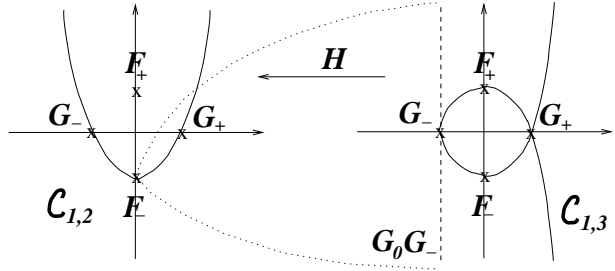} 
\par\end{centering}
\caption{\label{Excline}Exceptional line $G_0G_-$ on the algebraic transform $f_{1,2}\circ H=(b+a)f_{1,3}$.}
\end{figure}
The Heron transformation $H$ maps entire lines into points, so if $f$ happens to be $0$ at them then $f\circ H$ will surely be $0$ on the entire exceptional line. We can now explain the origin of the extraneous factors. Chebyshev curves never pass through $F_0$ (it is neither a triangle point nor the source point), but they do pass through $F_\pm$, which are the blow downs by $H$ of the lines $G_0G_\pm$ with equations $b\mp a=0$. These are exactly the factors that appear in 
$f_{p,q-p}\circ H$ in addition to $f_{p,q}$, see Fig. \ref{Excline}. Thus, $f_{p,q}$ is what from $f_{p,q-p}\circ H$ after dividing out these exceptional factors.
\begin{definition} Let $\CC_f$ be a projective algebraic curve defined by a homogeneous polynomial $f$. We call the curve $\overline{\CC}_f:=\CC_{f\circ H}$ the {\bf algebraic transform} of $\CC_f$. Furthermore, we denote by $\ff$ the polynomial obtained from $f\circ H$ by dividing out all the exceptional factors, and call the curve $\widetilde{\CC}_f:=\CC_{\ff}$ the {\bf proper transform} of $\CC_f$.
\end{definition}
Of course, $\ff$ divides $f\circ H$ so $\widetilde{\CC}_f\subseteq\overline{\CC}_f$ as sets, but at this point it is conceivable that the exceptional factors are not the only extraneous factors that can appear. The next theorem rules out this worry.
\begin{theorem}\label{HerProp} Let $p,q$ be relatively prime positive integers, and $q>p$. Then up to scale, $\ff_{p,q-p}=f_{p,q}^m$ for some $m\geq1$, and the only components of $\overline{\CC}_{p,q-p}$ other than $\CC_{p,q}$ (if any) are the exceptional lines $G_0G_\pm$ of $H$. It does contain them if and only if $\CC_{p,q-p}$ passes through the gate points $F_\pm$, respectively.
\end{theorem} 
\begin{proof}First we will show that $\widetilde{\CC}_{p,q-p}$ is a subset of $\CC_{p,q}$. Suppose $\xi\in\widetilde{\CC}_{p,q-p}$ is non-exceptional for $H$. Then $H(\xi)$ is also non-exceptional, and $f_{p,q}(H(\xi))=0$, i.e. $H(\xi)\in\CC_{p,q-p}$. Moreover, $\xi=H^{-1}\big(H(\xi)\big)$, and since $H^{-1}$ transforms $p:q-p$ triangles into $p:q$ triangles, we have that $\xi\in\CC_{p,q}$. There are at most finitely many  exceptional points on $\widetilde{\CC}_{p,q-p}$, because all exceptional factors are divided out, and on $\CC_{p,q}$, because with the exception of $\CC_{1,1}$ it has no straight line components, as the parametrization shows. Since both of them are algebraic curves we have $\widetilde{\CC}_{p,q-p}\subseteq\CC_{p,q}$ as sets. It follows from Hilbert's Nullstellensatz \cite[2.3.10]{Reid} that $\ff_{p,q-p}$ divides $f_{p,q}^n$ for some $n\geq1$. But $f_{p,q}$ is irreducible, so $\ff_{p,q-p}=f_{p,q}^m$ for some $m\geq1$ up to scale, and $\widetilde{\CC}_{p,q-p}$, $\CC_{p,q}$ coincide as sets.

If $\xi\in\overline{\CC}_{p,q-p}$, i.e. $(f_{p,q-p}\circ H)(\xi)=0$, and $H(\xi)$ is defined and non-exceptional, then $H(\xi)\in\CC_{p,q-p}$ and $\xi=H^{-1}\big(H(\xi)\big)\in\CC_{p,q}$. The only other possibilities are that $H(\xi)$ is undefined (i.e. $\xi$ is a base point of $H$) or $H(\xi)$ is a base point of $H^{-1}$ (other exceptional points are not in the image of $H$). There are only three base points, so the former case produces no components, and since $F_0\notin\CC_{p,q-p}$ in the latter case we must have $H(\xi)\in\{F_+,F_-\}$, so $\xi\in G_0G_+\cup G_0G_-$\,. Conversely, if $\CC_{p,q-p}$ passes through $F_+$ and/or $F_-$ then $f_{p,q-p}$ vanishes at them. Hence, its composition with $H$ vanishes on $G_0G_+$ and/or $G_0G_-$, respectively. Thus, they will be components of $\overline{\CC}_f$.
\end{proof}
As a matter of fact, $m=1$ and $f_{p,q}=\ff_{p,q-p}$, giving us the promised recursion. However, proving it, and determining the powers of the exceptional factors to be divided out of $f_{p,q-p}\circ H$ to get $\ff_{p,q-p}$, requires more work.

\section{Singularities and ``Pythagorean theorems"}\label{S5}

To relate $f_{p,q-p}\circ H$ to $\ff_{p,q-p}$ more precisely we need to know the powers of the exceptional factors, and we can expect from Theorem \ref{HerProp} that those depend on the behavior of $\CC_{p,q-p}$ at  $F_+$ and $F_-$. Indeed, it turns out that the powers of the factors are equal to the algebraic multiplicities of $F_\pm$ on $\CC_{p,q-p}$.
\begin{definition} Let $\CC_f$ be a projective algebraic curve and $\xi\in\CC_f$. Then the {\bf algebraic multiplicity} 
$\boldsymbol{m_f(\xi)}$ is the degree of the Taylor polynomial of $f$ at $\xi$ in some (and then any) affine coordinates. If $\xi\notin\CC_f$ we set $m_f(\xi):=0$. A point $\xi$ is called {\bf singular} (or a {\bf singularity}) if $m_f(\xi)\geq2$ \cite[6.2,10.2]{Gib}. 
\end{definition}
As an example, let us find the algebraic multiplicity of the source point $G_0$ on $\CC_{p,q}$. In $x=\frac{a}{c}$, $y=\frac{b}{c}$ coordinates $G_0$ is at the origin, and for $t\sim\infty$ we have:
\begin{equation*}
x=\frac{U_{p-1}(t)}{U_{p+q-1}(t)}\sim\frac{(2t)^{p-1}}{(2t)^{p+q-1}}=(2t)^{-q};\ \, 
y=\frac{U_{q-1}(t)}{U_{p+q-1}(t)}\sim\frac{(2t)^{q-1}}{(2t)^{p+q-1}}=(2t)^{-p}\,.
\end{equation*}
Therefore, the implicit equation of $\CC_{p,q}$ near the origin is $x^p-y^q=0$, up to the higher order terms. Thus, 
$m_{f_{p,q}}(G_0)=\min(p,q)$, and $G_0$ is singular whenever $p,q\geq2$. At self-intersection points the algebraic multiplicity is at least the number of (local) branches, but it can be strictly greater because of multiple tangents. In particular, if the real self-intersection multiplicity of a gate point in Table \ref{GateMult} is $2$ or more that point is also singular on $\CC_{p,q}$.

We will need some classical results concerning transformation of point multiplicities by quadratic Cremona transformations. They are usually formulated for the standard quadratic transformation $Q([a:b:c])=[bc:ac:ab]$ \cite[II.1.1]{Cool}, \cite[III.7.4]{Wal}, which was classically used to {\it resolve singularities} of plane curves. But $H$ can be obtained from $Q$ by composing with invertible linear transformations of $\C P^2$ that preserve all multiplicities, which makes it easy to rephrase results about $Q$ in terms directly applicable to $H$.
\begin{theorem}[{\bf \cite[III.7.4]{Wal}}]\label{TransMult} Let $\Phi$ be a quadratic Cremona transformation and $\CC_f$ be a degree $d$ plane algebraic curve with algebraic multiplicities $m_1$, $m_2$, $m_3$ at the base points of $\Phi^{-1}$. Then $\widetilde{\CC}_f$ is a degree $2d-m_1-m_2-m_3$ curve with the algebraic multiplicities $d-m_2-m_3$, $d-m_1-m_3$, $d-m_1-m_2$ at the base points of $\Phi$ (blown down from the exceptional lines of $\Phi^{-1}$ through its base points with the corresponding indices). The algebraic multiplicities of all non-exceptional points of $\CC_f$ are preserved on 
$\widetilde{\CC}_f$.
\end{theorem} 
If we could show that all $\CC_{p,q}$ can be obtained by iterating the Heron transformation starting from $\CC_{1,1}$, which is the line $b=a$, we would have a complete account of their singularities and their multiplicities, and hence of the exponents of the exceptional factors.

Unfortunately, this is not quite the case. We can go from $\CC_{p,q}$ to $\CC_{p,q+p}$, so we can get the chain of curves $1:1\to1:2\to1:3\to\dots$ through proper transforms by $H$, but we can never alter the first index. This is easily remedied, however.
\begin{definition} Let $S$ (for {\bf swap transformation}) denote the map $\C P^2\to\C P^2$ that transposes the first two homogeneous coordinates, $S([a:b:c])=[b:a:c]$. Clearly, $S=S^{-1}$. For curves we define $S(\CC_f)=\CC_{f\circ S^{-1}}=\CC_{f\circ S}$.
\end{definition}
\begin{figure}[!ht]
\begin{equation*}
\begin{array}{ccccccccccc}
1:1 &  \xrightarrow[]{H} & 1:2 &  \xrightarrow[]{H} &   1:3 &  \xrightarrow[]{H} & 1:4 &  \xrightarrow[]{H} & 1:5 &  \xrightarrow[]{H} & \dots  \\
  &  & \vdots &   &  \downarrow\lefteqn{S}  &   & \vdots &   & \vdots &   &   \\
 &   &  & & 3:1 & \xrightarrow[]{H}   &  3:4 & \xrightarrow[]{H} &  3:7 & \xrightarrow[]{H} &  \dots     \\
 &   &  &   & \vdots   &   & \downarrow\lefteqn{S}  &   & \vdots &   &   \\
 &   &  & &  &    &  4:3 & \xrightarrow[]{H} &  4:7 & \xrightarrow[]{H} &  \dots     \\
\end{array}
\end{equation*}
\caption{\label{EucRa} Euclidean algorithm recursion on ratios.}
\end{figure}
Obviously, $S(\CC_{p,q})=\CC_{q,p}$, and, being linear and invertible, $S$ preserves all the algebraic multiplicities. Now we can swap $1:3$ into $3:1$, and obtain some $3:q$ curves. At any point the swap can be used again to change the first index, etc. Clearly, many more $\CC_{p,q}$ can be generated in this way, Fig. \ref{EucRa}. But can we generate them all? It turns out that we can, and the requisite sequence of transformations is obtained by performing the Euclidean algorithm on $q,p$.

Let $p,q$ be relatively prime, $q>p$. According to the Euclidean algorithm, 
\begin{align*}
&q=d_1p+r_1\\
&p=d_2r_1+r_2\\
&\ \ \ \ \ \ \vdots\\
&r_{n-2}=d_nr_{n-1}+1\\
&r_{n-1}=d_{n+1}1+0\,.
\end{align*}
So $H^{d_1}$ takes us from $p:r_1$ to $p:r_1+d_1p=p:q$. The swap gets us $p:r_1$ from $r_1:p$, which in turn we get from $r_1:r_2$ by applying $H^{d_2}$, and so on. Since there is no $\CC_{1,0}$ we rewrite the last line as $r_{n-1}=(d_{n+1}-1)1+1$ to start from $\CC_{1,1}$. Thus, we can get $\CC_{p,q}$ by applying 
\begin{equation}\label{HerRec}
H^{d_1}SH^{d_2}S\dots H^{d_n}SH^{d_{n+1}-1}
\end{equation}
to $\CC_{1,1}$, where each application of $H$ means taking the zero set of the proper transform by $H$. To track the multiplicities of potential singularities through these transformations we will look at how $H$ and $S$ transform the gate points, and apply Theorem \ref{TransMult}.
\begin{theorem}[{\bf Pythagorean recursion}]\label{TPythIt} For all relatively prime positive integers $p,q$ with $q>p$ we have $\CC_{p,q}=\widetilde{\CC}_{p,q-p}$. All singularities of $\CC_{p,q}$, if any, are at the gate points and/or at the source point. The algebraic multiplicity of the source point is $\min(p,q)$, and the algebraic multiplicities of the gate points are equal to their real self-intersection multiplicities given in Table \ref{GateMult}. Let $f_{p,q}$ be the implicit equation of $\CC_{p,q}$, and $H$ be the Heron transformation. Then up to scale
\begin{equation}\label{PythIt}
f_{p,q-p}\circ H=(b+a)^{m_{p,q}^+}(b-a)^{m_{p,q}^-}f_{p,q}\,,
\end{equation}
where for $q-p$ odd $m_{p,q}^\pm=\frac{q-p-1}2$, while for $q-p$ even $m_{p,q}^+=\frac{q-p}2$, and $m_{p,q}^-=\frac{q-p}2-1$.
\end{theorem} 
\begin{proof}Let $\mathcal{H}_{p,q}$ denote the curves obtained from $\CC_{1,1}$ by composing proper transforms by $H$ and $S$ as in \eqref{HerRec}. We will prove that $\mathcal{H}_{p,q}=\CC_{p,q}$, along with the claims about multiplicities, by induction on their application. Both claims are true for $\CC_{1,1}$ since it is a straight line passing through $G_+$ and $G_0$, and no other gate points. Now suppose that they hold for $\CC_{p,q}=\mathcal{H}_{p,q}$, we need to show the same for $\CC_{p,q+p}$ and $\CC_{q,p}$. 

For $\mathcal{H}_{p,q+p}=\widetilde{\CC}_{p,q}$ there are three cases: when $p,q$ are both odd, and when $p$ or $q$ is even (they are relatively prime, so not both even). We will only check the first case, the other two are analogous. Applying the transformation formulas from Theorem \ref{TransMult} to the first row of Table \ref{GateMult} we get values that should match the values in the second row of Table \ref{GateMult} with $q$ replaced by 
$p+q$ since $p+q$ is even. And they do. The identity \eqref{PythIt}, with $f_{p,q}$ replaced by $\widetilde{f}_{p,q-p}$, follows directly from Theorem \ref{TransMult}, with $m_{p,q}^\pm$ being the multiplicities of $F_{\mp}$ on $\CC_{p,q-p}$, which we expressed explicitly according to Table \ref{GateMult}. In all cases $m_{p,q}^++m_{p,q}^-=q-p-1$, so 
$$
\deg(\widetilde{f}_{p,q-p})=2\deg\widetilde{f}_{p,q-p}-(q-p-1)=2(p+q-p-1)-(q-p-1)=\deg(f_{p,q}).
$$
But by Theorem \ref{HerProp} we have $f_{p,q}^m=\widetilde{f}_{p,q-p}$. Hence, $m=1$ and $f_{p,q}=\widetilde{f}_{p,q-p}$. Applying this to $f_{p,q+p}$ we get $f_{p,q+p}=\widetilde{f}_{p,q}$ and $\CC_{p,q+p}=\mathcal{H}_{p,q+p}$. 

Since the swap preserves the degrees of curves $\mathcal{H}_{q,p}=\CC_{q,p}$ is obvious. In terms of Table \ref{GateMult}, applying $S$ amounts to transposing $p$ and $q$, $E_{\pm}$ and $F_{\pm}$ columns (because $S$ swaps those points), and the second and the third rows (because the parities of $p$ and $q$ are also swapped). As one can check, the combination of these moves leaves Table \ref{GateMult} intact. Hence, $\CC_{q,p}$ also has the multiplicities given by it. This completes the induction.

Since the algebraic multiplicities of the gate points are fully accounted for by the real self-intersecting branches with distinct tangents all non-base exceptional points of $\CC_{p,q}$ are non-singular. The non-exceptional points are also non-singular since their multiplicities are preserved. Thus, starting from $\CC_{1,1}$ the proper transform by $H$ can only create singularities at $G_0$, $G_{\pm}$, and then move the latter to $F_{\mp}$. From there $S$ can also move them to 
$E_{\mp}$, which it swaps with $F_{\mp}$. But $S$ fixes $G_0$, $G_{\pm}$, and $H$ fixes $E_{\pm}$, so no further singularities are created by this process. Since the multiplicity of $G_0$ was computed earlier this concludes the proof.
\end{proof}
Note that the iteration based on \eqref{PythIt} is not only more effective than computing resultants, but also more so than the naive instruction to divide out all the exceptional factors from $f_{p,q}\circ H$. In the affine coordinates $x=\frac{b}{a}$, $y=\frac{c}{a}$ the exceptional factors become $x\pm1$, and since  \eqref{PythIt} gives us their exact multiplicities we can divide them out by the ordinary long division in one variable.

\bigskip

\noindent {\small {\bf Acknowledgements:} The author is grateful to R.G. Rogers for explaining the history of ideas developed in this paper, and to D. Karp for discussions about the geometry of plane Cremona transpformations.}

\end{document}